\documentclass[11pt,a4paper]{amsart}

\usepackage[T1]{fontenc}
\usepackage[utf8]{inputenc}
\usepackage{lmodern}
\usepackage[a4paper,margin=3cm]{geometry}
\usepackage{microtype}
\usepackage{amsmath,amssymb,amsthm,mathtools}
\usepackage{enumitem}
\usepackage{graphicx,booktabs}
\usepackage[hidelinks]{hyperref}
\numberwithin{equation}{section}

\newtheorem{theorem}{Theorem}[section]
\newtheorem{proposition}[theorem]{Proposition}
\newtheorem{lemma}[theorem]{Lemma}
\newtheorem{corollary}[theorem]{Corollary}
\theoremstyle{definition}
\newtheorem{definition}[theorem]{Definition}
\newtheorem{remark}[theorem]{Remark}

\newcommand{\R}{\mathbb{R}}
\newcommand{\T}{\mathbb{T}}
\newcommand{\EE}{\mathcal{E}}
\newcommand{\JJ}{\mathcal{J}}
\newcommand{\HH}{\mathcal{H}}
\newcommand{\ddmu}{\,d\mu}
\newcommand{\abs}[1]{\lvert#1\rvert}
\newcommand{\norm}[1]{\lVert#1\rVert}
\newcommand{\PV}{\operatorname{P.V.}}
\newcommand{\cc}{\mathrm{cc}}
\newcommand{\cdt}{\mathrm{cd}}
\newcommand{\dd}{\mathrm{dd}}
\newcommand{\avg}[1]{\langle #1\rangle_\mu}

\title[Fractional variational problem on a hybrid time scale]{A fractional variational problem on a continuous-discrete time scale: coercivity and weak solvability}

\author{Hafida Abbas}
\address{Department of Economics, University Dr Moulay Tahar, Saida, Algeria}
\email{hafida.abbas@univ-saida.dz}

\author{Abdelhalim Azzouz}
\address{Department of Mathematics, University Salhi Ahmed, Naama, Algeria}
\email{azzouz.abdelhalim@cuniv-naama.dz}

\author{Gauhar Rahman}
\address{Department of Mathematics and Statistics, Hazara University, Mansehra, Pakistan}
\email{drgauhar.rahman@hu.edu.pk}

\author{Khidir Shaib Mohamed}
\address{Department of Mathematics, College of Sciences, Qassim University, Buraydah 51452, Saudi Arabia}
\email{k.idris@qu.edu.sa}


\begin{document}

\begin{abstract}
We study a fractional variational problem on a continuous-discrete time scale $\T=[0,1]\cup\{2,\dots,N+1\}$, built from a Gagliardo-type nonlocal energy with respect to the Lebesgue--Dirac measure. We show that the interface term, although of lower order with respect to regularity, couples the two phases coercively: it yields a hybrid Poincar\'e inequality and forces the energy to vanish only on global constants. With a positive zero-order term, the problem has a unique weak solution, characterized by a coupled Euler--Lagrange system in which a nonlocal integral equation on the interval and a finite nonlocal algebraic system on the nodes communicate through the interface. Without the zero-order term, the problem becomes of Neumann type: it is solvable under a single global compatibility condition rather than one per phase, and its solution, unique up to constants, is the limit of the regularized solutions at a linear rate. A reproducible conforming Galerkin study quantifies the discrete spectral gaps, examines ablations of the three energy contributions, and checks the global compatibility mechanism and the regularization rate.
\end{abstract}

\maketitle

\section{Introduction}

Nonlocal interactions can transmit information between components that are separated in the ambient Euclidean topology. This observation raises a particularly transparent variational question when one component is an interval and the other consists of finitely many isolated nodes: which interactions ensure a single global equilibrium, and which are needed to retain fractional regularity on the continuous component? The answer depends on the structure of the energy, not merely on the presence of a continuous and a discrete set.

The underlying analytical ingredients come from two related lines of work. Fractional Sobolev spaces on ordinary domains are naturally described by Gagliardo difference quotients \cite{Gagliardo1957,Slobodeckij1958,DiNezza2012,MRS2016}, while time-scale calculus provides a common language for continuous and discrete structures \cite{Hilger1990,BohnerPeterson2001}. Measure-theoretic integration and Sobolev spaces on time scales have been studied in \cite{BohnerGuseinov2006,CabadaVivero2006,AgarwalOteroPereraVivero2006,SkrzypekSzymanska2019}. Fractional spaces associated with Riemann--Liouville or conformable derivatives appear in \cite{WangZhouLi2016,HuLi2022JFS,HuLi2022Right}. A recent Gagliardo-type theory on arbitrary time scales, with the \emph{canonical} Lebesgue $\Delta$-measure, develops functional properties and geometric inequalities in a much broader setting \cite{AbbasAzzouzAgarwalTorres2026}. The present problem is different in two explicit respects: it uses a prescribed Lebesgue--Dirac measure with unit atomic masses, rather than the canonical $\Delta$-measure, and it studies the variational and spectral consequences of a specific continuous--atomic coupling.

Nonlocal coupling and transmission have also been investigated independently of time-scale calculus. Energy-based local--nonlocal formulations and their existence theory were developed by Acosta, Bersetche and Rossi \cite{AcostaBersetcheRossi2022}; their Neumann-type coupling models already exhibit uniqueness modulo global constants \cite{AcostaBersetcheRossi2023}. Related interface energies and their numerical treatment appear in \cite{CapodaglioDeliaBochevGunzburger2020,BorthagarayCiarlet2026}. Particularly close in spirit, Bersetche, Ot\'arola and Quero \cite{BersetcheOtarolaQuero2026} analyze two disjoint continuum domains interacting through a nonlocal kernel and develop finite-element and domain-decomposition methods. These works show that neither coercive coupling nor a single compatibility condition should be presented as new phenomena in general. Their settings, however, do not consist of an interval coupled to finitely many atoms under the mixed measure considered here.

For Neumann-type nonlocal problems, the variational theory, compatibility conditions and Poincar\'e inequalities form a substantial existing literature: see Frerick, Vollmann and Vu \cite{FrerickVollmannVu2025} for general symmetric measurable kernels, and Kreisbeck and Sch\"onberger \cite{KreisbeckSchoenberger2024} for a different nonlocal-gradient setting in which the null space can even be infinite dimensional. Finite-element approximation for fractional Neumann problems has been studied in \cite{BersetcheBorthagaray2022}. These comparisons clarify the question addressed below: for \emph{our particular Gagliardo difference energy}, which components of the quadratic form eliminate the relative-constant mode, and what remains true when the zero-order regularization is removed?

We fix $N\geq2$ and consider the compact hybrid set
\[
\T=[0,1]\cup D, \qquad D=\{2,3,\ldots,N+1\},
\]
with the explicitly prescribed measure
\[
\mu=\mathbf{1}_{[0,1]}(x)\,dx+\sum_{k=2}^{N+1}\delta_k.
\]
Our energy is the Gagliardo-type double integral with kernel $|x-y|^{-1-2\alpha}$, $0<\alpha<1$, defined off the diagonal. It decomposes as
\[
\EE_\alpha(u)=\EE_{\cc}(v)+2\EE_{\cdt}(v,a)+\EE_{\dd}(a).
\]
The continuous--continuous term supplies the $H^\alpha(0,1)$ regularity, the mixed term connects the nodal values with the interval, and the discrete--discrete term reinforces interactions between nodes. In this model the mixed term has a uniform positive lower bound on the continuous--atomic pairs, making it possible to distinguish the three roles quantitatively.

An earlier two-author preprint of this problem \cite{AbbasAzzouz2026Transmission} established the decomposition, interface coercivity and the positive-zero-order variational formulation. The present manuscript is a \emph{revised and extended version of that work}, rather than an independent rediscovery of those results. Its additional analysis concerns the degenerate problem, structural necessity of the energy components, and a reproducible Galerkin study. The extension has three connected parts.

First, for $\lambda=0$ we characterize solvability by the single global balance condition
\[
\int_{\T}f\,d\mu=\int_0^1 f(x)\,dx+\sum_{k=2}^{N+1}f(k)=0,
\]
and obtain uniqueness modulo constants. If the interface interaction is removed, two independent compatibility conditions are required. The compatible solutions of the regularized problems converge to the zero-mean solution with a linear $O(\!\lambda)$ rate in the hybrid energy norm. These conclusions are specific realizations of familiar variational principles; their value here is the explicit dependence on the continuous--atomic transmission mechanism.

Second, we distinguish \emph{necessity} from \emph{quantitative strengthening}: the mixed term is essential for connecting the two phases, and the continuous--continuous term is essential for uniform fractional regularity control, whereas the discrete--discrete term is not needed for qualitative coercivity on this finite atomic component. Third, a reproducible conforming Galerkin study measures spectral gaps, performs energy-term ablations, probes the global compatibility condition, and checks the vanishing-regularization rate. The computations are not offered as proofs of continuous spectral convergence or sharp constants; they provide quantitative diagnostics of the mechanisms proved analytically.

Section~\ref{sec:coercive-framework} presents the mixed-measure energy and coercivity estimates. Section~\ref{sec:variational} develops the positive-parameter variational problem and the coupled Euler--Lagrange equations. Section~\ref{sec:neumann} treats the degenerate Neumann-type regime and the regularization limit. Section~\ref{sec:numerical} gives the Galerkin protocol, ablation results, and numerical limitations, followed by the conclusions.

\section{The hybrid setting and the main coercivity result}\label{sec:coercive-framework}
While the measure-theoretic and first-order Sobolev frameworks on time scales are already available, our purpose here is different: we introduce a nonlocal Gagliardo-type energy on the continuous-discrete time scale under consideration and use it to formulate the fractional variational problem studied in the sequel.
\subsection{The hybrid domain and the nonlocal energy}

Fix an integer $N\ge 2$ and define
\[
\T=[0,1]\cup D,
\qquad D:=\{2,3,\dots,N+1\}.
\]
We equip $\T$ with the following positive finite Borel measure, consisting of Lebesgue measure on the interval and unit Dirac masses on the discrete nodes:
\begin{equation}\label{eq:def-mu}
\mu(A):=\int_{A\cap[0,1]}dx+\sum_{k=2}^{N+1}\delta_k(A),
\end{equation}
for every Borel set $A\subset \T$. This is a chosen Lebesgue--Dirac measure and should not be confused with the canonical Lebesgue $\Delta$-measure of $\T$, whose atomic masses are determined by the forward jump operator. Equivalently,
\[
\int_{\T} f\ddmu =\int_0^1 f(x)\,dx+\sum_{k=2}^{N+1} f(k)
\]
for every integrable function $f$ on $\T$.

A function $u:\T\to\R$ will be written as
\[
u=(v,a_2,\dots,a_{N+1}),
\]
where
\[
v=u|_{[0,1]},\qquad a_k=u(k)\quad (k=2,\dots,N+1).
\]
Thus the continuous and discrete phases are encoded separately.

Let $\alpha\in(0,1)$. To account for the atoms of $\mu$, we define the kernel on the entire product $\T\times\T$ by
\[
K_\alpha(x,y):=
\begin{cases}
\abs{x-y}^{-1-2\alpha}, & x\neq y,\\
0, & x=y.
\end{cases}
\]
For $u\in L^2(\T,\mu)$, define the (possibly infinite) quadratic energy
\begin{equation}\label{eq:def-energy}
\EE_\alpha(u):=
\iint_{\T\times\T}
K_\alpha(x,y)\abs{u(x)-u(y)}^2
\,d\mu(x)\,d\mu(y)\in[0,+\infty].
\end{equation}
The diagonal convention is essential at the discrete nodes, where
$(\mu\otimes\mu)(\{(k,k)\})=1$; on the continuous--continuous part the diagonal has zero product measure. Equivalently, the integral may be taken over $(\T\times\T)\setminus\{(x,x):x\in\T\}$.
Since the discrete phase is finite and separated from $[0,1]$ by a positive distance, the mixed and discrete contributions are automatically finite whenever the restriction $v$ belongs to $H^\alpha(0,1)$.

\begin{definition}
We define the hybrid energy space by
\[
\HH_\alpha(\T):=
\Big\{u=(v,a_2,\dots,a_{N+1}): v\in H^\alpha(0,1),\ a_k\in\R\Big\}.
\]
It is endowed with the norm
\begin{equation}\label{eq:def-Hnorm}
\norm{u}_{\HH_\alpha(\T)}^2:=\norm{u}_{L^2(\T,\mu)}^2+\EE_\alpha(u).
\end{equation}
\end{definition}

The continuous--continuous part of \eqref{eq:def-energy} is the standard Gagliardo seminorm squared on $(0,1)$, namely
\begin{equation}\label{eq:gagliardo-seminorm}
[v]_{H^\alpha(0,1)}^2:=\int_0^1\int_0^1\frac{\abs{v(x)-v(y)}^2}{\abs{x-y}^{1+2\alpha}}\,dx\,dy.
\end{equation}
\begin{lemma}[Equivalence with the product norm]\label{lem:product-norm}
There exist constants $c_0,C_0>0$, depending only on $\alpha$ and $N$, such that for every
\[u=(v,a_2,\dots,a_{N+1})\in \HH_\alpha(\T),
\]
one has
\[
c_0\Bigl(\norm{v}_{H^\alpha(0,1)}^2+\sum_{k=2}^{N+1}\abs{a_k}^2\Bigr)
\le \norm{u}_{\HH_\alpha(\T)}^2
\le C_0\Bigl(\norm{v}_{H^\alpha(0,1)}^2+\sum_{k=2}^{N+1}\abs{a_k}^2\Bigr).
\]
Consequently, $\HH_\alpha(\T)$ is a Hilbert space.
\end{lemma}

\begin{proof}
The upper bound follows from the decomposition in Proposition~\ref{prop:decomposition}. The continuous--continuous term is exactly $[v]_{H^\alpha(0,1)}^2$. Since $D$ is finite and $\operatorname{dist}([0,1],D)=1$, one has
\[
1\le \abs{k-x}\le N+1\qquad (x\in[0,1],\ k\in D),
\]
which gives
\[
\EE_{\cdt}(v,a)
\le \sum_{k=2}^{N+1}\int_0^1 \abs{a_k-v(x)}^2\,dx
\le 2\sum_{k=2}^{N+1}\abs{a_k}^2+2N\norm{v}_{L^2(0,1)}^2.
\]
Similarly, since $D$ is finite,
\[
\EE_{\dd}(a)
\le 2\sum_{\substack{i,j=2\\ i\neq j}}^{N+1}(\abs{a_i}^2+\abs{a_j}^2)
\le 4N\sum_{k=2}^{N+1}\abs{a_k}^2.
\]
Hence
\[
\EE_\alpha(u)\le C\Bigl([v]_{H^\alpha(0,1)}^2+\norm{v}_{L^2(0,1)}^2+\sum_{k=2}^{N+1}\abs{a_k}^2\Bigr)
\]
for some $C>0$, which yields the upper bound.

For the lower bound, we simply note that
\[
\norm{u}_{\HH_\alpha(\T)}^2=\norm{u}_{L^2(\T,\mu)}^2+\EE_\alpha(u)
\ge \norm{u}_{L^2(\T,\mu)}^2
=\norm{v}_{L^2(0,1)}^2+\sum_{k=2}^{N+1}\abs{a_k}^2
\]
and also $\EE_\alpha(u)\ge [v]_{H^\alpha(0,1)}^2$. Combining these estimates gives
\[
\norm{u}_{\HH_\alpha(\T)}^2
\ge \norm{v}_{L^2(0,1)}^2+[v]_{H^\alpha(0,1)}^2+\sum_{k=2}^{N+1}\abs{a_k}^2
=\norm{v}_{H^\alpha(0,1)}^2+\sum_{k=2}^{N+1}\abs{a_k}^2,
\]
which proves the claim.
\end{proof}

\subsection{Decomposition of the energy and the interface estimate}

The energy \eqref{eq:def-energy} splits naturally into three contributions.

\begin{proposition}[Decomposition of the energy]\label{prop:decomposition}
For every $u=(v,a_2,\dots,a_{N+1})\in \HH_\alpha(\T)$,
\begin{equation}\label{eq:energy-decomposition}
\EE_\alpha(u)=\EE_{\cc}(v)+2\EE_{\cdt}(v,a)+\EE_{\dd}(a),
\end{equation}
where
\begin{align}
\EE_{\cc}(v)&:=\int_0^1\int_0^1\frac{\abs{v(x)-v(y)}^2}{\abs{x-y}^{1+2\alpha}}\,dx\,dy,
\label{eq:Ecc}\\
\EE_{\cdt}(v,a)&:=\sum_{k=2}^{N+1}\int_0^1\frac{\abs{a_k-v(x)}^2}{\abs{k-x}^{1+2\alpha}}\,dx,
\label{eq:Ecd}\\
\EE_{\dd}(a)&:=\sum_{\substack{i,j=2\\ i\neq j}}^{N+1}\frac{\abs{a_i-a_j}^2}{\abs{i-j}^{1+2\alpha}}.
\label{eq:Edd}
\end{align}
\end{proposition}

\begin{proof}
Using \eqref{eq:def-mu}, the product measure $\mu\otimes\mu$ decomposes into the sum of the continuous--continuous, continuous--discrete, discrete--continuous, and discrete--discrete parts. Substituting this decomposition into \eqref{eq:def-energy}, we obtain
\begin{align*}
\EE_\alpha(u)
&=\int_0^1\int_0^1\frac{\abs{v(x)-v(y)}^2}{\abs{x-y}^{1+2\alpha}}\,dx\,dy \\
&\quad +\sum_{k=2}^{N+1}\int_0^1\frac{\abs{a_k-v(x)}^2}{\abs{k-x}^{1+2\alpha}}\,dx
+\sum_{k=2}^{N+1}\int_0^1\frac{\abs{v(x)-a_k}^2}{\abs{x-k}^{1+2\alpha}}\,dx \\
&\quad +\sum_{\substack{i,j=2\\ i\neq j}}^{N+1}\frac{\abs{a_i-a_j}^2}{\abs{i-j}^{1+2\alpha}}.
\end{align*}
Since $\abs{k-x}=\abs{x-k}$ and $\abs{a_k-v(x)}=\abs{v(x)-a_k}$, the two mixed terms coincide, which yields \eqref{eq:energy-decomposition}.
\end{proof}

\begin{lemma}[Interface estimate]\label{lem:interface-estimate}
There exist constants $c_1,c_2>0$, depending only on $\alpha$ and $N$, such that for every $u=(v,a_2,\dots,a_{N+1})\in\HH_\alpha(\T)$,
\begin{equation}\label{eq:interface-equivalence}
c_1\sum_{k=2}^{N+1}\int_0^1\abs{a_k-v(x)}^2\,dx
\le \EE_{\cdt}(v,a)
\le c_2\sum_{k=2}^{N+1}\int_0^1\abs{a_k-v(x)}^2\,dx.
\end{equation}
\end{lemma}

\begin{proof}
Fix $k\in\{2,\dots,N+1\}$ and $x\in[0,1]$. Then
\[
k-1\le k-x\le k.
\]
Therefore,
\[
\frac{1}{k^{1+2\alpha}}\le \frac{1}{\abs{k-x}^{1+2\alpha}}\le \frac{1}{(k-1)^{1+2\alpha}}.
\]
Multiplying by $\abs{a_k-v(x)}^2$ and integrating over $(0,1)$, we obtain
\[
\frac{1}{k^{1+2\alpha}}\int_0^1 \abs{a_k-v(x)}^2\,dx
\le
\int_0^1 \frac{\abs{a_k-v(x)}^2}{\abs{k-x}^{1+2\alpha}}\,dx
\le
\frac{1}{(k-1)^{1+2\alpha}}\int_0^1 \abs{a_k-v(x)}^2\,dx.
\]
Summing with respect to $k$ yields \eqref{eq:interface-equivalence} with
\[
c_1=(N+1)^{-(1+2\alpha)},
\qquad
c_2=1.
\]
\end{proof}

\begin{lemma}[Algebraic decomposition around the mean]\label{lem:mean-decomposition}
Let
\[
\bar v:=\int_0^1 v(x)\,dx.
\]
Then, for every $k\in\{2,\dots,N+1\}$,
\begin{equation}\label{eq:single-mean-decomposition}
\int_0^1 \abs{a_k-v(x)}^2\,dx
=\abs{a_k-\bar v}^2+\norm{v-\bar v}_{L^2(0,1)}^2.
\end{equation}
Consequently,
\begin{equation}\label{eq:sum-mean-decomposition}
\sum_{k=2}^{N+1}\int_0^1 \abs{a_k-v(x)}^2\,dx
=
\sum_{k=2}^{N+1}\abs{a_k-\bar v}^2
+N\norm{v-\bar v}_{L^2(0,1)}^2.
\end{equation}
\end{lemma}

\begin{proof}
Fix $k\in\{2,\dots,N+1\}$. We write
\[
a_k-v(x)=(a_k-\bar v)+(\bar v-v(x)).
\]
Squaring and integrating over $(0,1)$, we obtain
\begin{align*}
\int_0^1\abs{a_k-v(x)}^2\,dx
&=\int_0^1 \abs{a_k-\bar v}^2\,dx
+\int_0^1\abs{v(x)-\bar v}^2\,dx \\
&\quad +2(a_k-\bar v)\int_0^1(\bar v-v(x))\,dx.
\end{align*}
Since
\[
\int_0^1(\bar v-v(x))\,dx=\bar v-\int_0^1v(x)\,dx=0,
\]
the mixed term vanishes. This proves \eqref{eq:single-mean-decomposition}. Summing over $k$ gives \eqref{eq:sum-mean-decomposition}.
\end{proof}

\subsection{Interface coercivity and the hybrid energy space}

We now derive the main coercivity statement.

\begin{theorem}[Interface coercivity]\label{thm:interface-coercivity}
There exist constants $C_1,C_2>0$, depending only on $\alpha$ and $N$, such that for every $u=(v,a_2,\dots,a_{N+1})\in\HH_\alpha(\T)$, setting
\[
\mathcal M(u):=[v]_{H^\alpha(0,1)}^2+\norm{v-\bar v}_{L^2(0,1)}^2+\sum_{k=2}^{N+1}\abs{a_k-\bar v}^2,
\]
one has
\begin{equation}\label{eq:coercivity-twosided}
C_1\,\mathcal M(u)\le \EE_\alpha(u)\le C_2\,\mathcal M(u).
\end{equation}
\end{theorem}

\begin{proof}
By Proposition~\ref{prop:decomposition},
\[
\EE_\alpha(u)=\EE_{\cc}(v)+2\EE_{\cdt}(v,a)+\EE_{\dd}(a).
\]
Since all three terms on the right-hand side are nonnegative, Lemmas~\ref{lem:interface-estimate} and \ref{lem:mean-decomposition} imply
\begin{align*}
\EE_\alpha(u)
&\ge \EE_{\cc}(v)+2\EE_{\cdt}(v,a) \\
&\ge [v]_{H^\alpha(0,1)}^2+2c_1\sum_{k=2}^{N+1}\int_0^1\abs{a_k-v(x)}^2\,dx \\
&= [v]_{H^\alpha(0,1)}^2
+2c_1\sum_{k=2}^{N+1}\abs{a_k-\bar v}^2
+2c_1N\norm{v-\bar v}_{L^2(0,1)}^2.
\end{align*}
This proves the lower bound in \eqref{eq:coercivity-twosided}.

For the reverse estimate, we use again Proposition~\ref{prop:decomposition}. The term $\EE_{\cc}(v)$ is exactly $[v]_{H^\alpha(0,1)}^2$. By Lemmas~\ref{lem:interface-estimate} and \ref{lem:mean-decomposition},
\[
\EE_{\cdt}(v,a)
\le c_2\Big(
\sum_{k=2}^{N+1}\abs{a_k-\bar v}^2+N\norm{v-\bar v}_{L^2(0,1)}^2
\Big).
\]
It remains to estimate the discrete--discrete contribution. For $i\neq j$,
\[
\abs{a_i-a_j}^2
\le 2\abs{a_i-\bar v}^2+2\abs{a_j-\bar v}^2.
\]
Since $\sum_{j\neq i}\abs{i-j}^{-1-2\alpha}\le 2\sum_{m=1}^{N-1}m^{-1-2\alpha}$ for every $i\in D$, it follows that
\begin{align*}
\EE_{\dd}(a)
&\le 2\sum_{\substack{i,j=2\\ i\neq j}}^{N+1}
\frac{\abs{a_i-\bar v}^2+\abs{a_j-\bar v}^2}{\abs{i-j}^{1+2\alpha}} \\
&\le 8\Big(\sum_{m=1}^{N-1}\frac{1}{m^{1+2\alpha}}\Big)
\sum_{k=2}^{N+1}\abs{a_k-\bar v}^2.
\end{align*}
Combining these estimates gives the upper bound in \eqref{eq:coercivity-twosided}.
\end{proof}

\begin{remark}\label{rem:interface-role}
By Lemma~\ref{lem:interface-estimate}, the interface term is bounded by $L^2$-type quantities, so it is of lower order with respect to regularity. Theorem~\ref{thm:interface-coercivity} shows that it nevertheless plays a structural role: it controls the $L^2$ oscillation of the continuous component and pins the discrete variables to the average value of the continuous phase. As shown in Section~\ref{sec:neumann}, this is precisely what is needed to solve the degenerate problem $\lambda=0$, where no zero-order term is available.
\end{remark}

\begin{corollary}[Hybrid Poincar\'e-type inequality]\label{cor:hybrid-poincare}
There exists a constant $C_P>0$, depending only on $\alpha$ and $N$, such that for every $u=(v,a_2,\dots,a_{N+1})\in \HH_\alpha(\T)$,
\[
\norm{v-\bar v}_{L^2(0,1)}^2+\sum_{k=2}^{N+1}\abs{a_k-\bar v}^2
\le C_P\,\EE_\alpha(u).
\]
One can take $C_P=\tfrac12(N+1)^{1+2\alpha}$.
\end{corollary}

\begin{proof}
By the proof of the lower bound in \eqref{eq:coercivity-twosided} and since $N\ge1$,
\[
\EE_\alpha(u)\ge 2c_1\Big(N\norm{v-\bar v}_{L^2(0,1)}^2+\sum_{k=2}^{N+1}\abs{a_k-\bar v}^2\Big)
\ge 2c_1\Big(\norm{v-\bar v}_{L^2(0,1)}^2+\sum_{k=2}^{N+1}\abs{a_k-\bar v}^2\Big),
\]
with $c_1=(N+1)^{-(1+2\alpha)}$ from Lemma~\ref{lem:interface-estimate}.
\end{proof}

A direct consequence is that the quadratic form
\[
\|u\|_{\ast}^2:=\|u\|_{L^2(\T,\mu)}^2+ [v]_{H^\alpha(0,1)}^2+\norm{v-\bar v}_{L^2(0,1)}^2+\sum_{k=2}^{N+1}\abs{a_k-\bar v}^2
\]
defines a norm equivalent to \eqref{eq:def-Hnorm} on $\HH_\alpha(\T)$.

\section{Variational solvability and the coupled Euler--Lagrange system}\label{sec:variational}
We now formulate the fractional variational problem in weak form and prove its well-posedness for $\lambda>0$, in line with the variational treatment of nonlocal problems; see, for instance, \cite{AndreuMazonRossiToledo2010,MRS2016}. For $\lambda>0$ this is a direct application of the direct method of the calculus of variations; the degenerate case $\lambda=0$, where the interface coercivity becomes indispensable, is treated in Section~\ref{sec:neumann}.
\subsection{Formulation of the problem}

Let $\lambda>0$ and let $f\in L^2(\T,\mu)$. We consider the quadratic functional
\begin{equation}\label{eq:def-functional}
\JJ(u):=\frac12\EE_\alpha(u)+\frac{\lambda}{2}\norm{u}_{L^2(\T,\mu)}^2-\int_{\T}f u\ddmu,
\qquad u\in\HH_\alpha(\T).
\end{equation}

The bilinear form naturally associated with the energy is
\begin{equation}\label{eq:def-bilinear}
\mathfrak a(u,\varphi):=
\iint_{\T\times\T}
K_\alpha(x,y)(u(x)-u(y))(\varphi(x)-\varphi(y))
\,d\mu(x)\,d\mu(y).
\end{equation}
Note that $\mathfrak a(u,u)=\EE_\alpha(u)$. By the Cauchy--Schwarz inequality,
\[
\abs{\mathfrak a(u,\varphi)}\le
\EE_\alpha(u)^{1/2}\EE_\alpha(\varphi)^{1/2}.
\]

\begin{definition}
We say that $u\in\HH_\alpha(\T)$ is a \emph{weak solution} of the hybrid nonlocal problem associated with $(\lambda,f)$ if
\begin{equation}\label{eq:weak-solution}
\mathfrak a(u,\varphi)+\lambda\int_{\T}u\varphi\ddmu=\int_{\T}f\varphi\ddmu
\qquad\text{for every }\varphi\in\HH_\alpha(\T).
\end{equation}
\end{definition}

Equation \eqref{eq:weak-solution} is the weak form of a coupled nonlocal problem on the continuous-discrete time scale: the continuous phase is driven by a nonlocal integral operator, the discrete phase by a finite algebraic system, and the two parts are coupled through the interface term, which acts as a transmission mechanism between them.

\subsection{Existence, uniqueness, and variational characterization}

\begin{theorem}[Existence and uniqueness of the minimizer]\label{thm:existence-uniqueness}
For every $\lambda>0$ and every $f\in L^2(\T,\mu)$, the functional $\JJ$ admits a unique minimizer in $\HH_\alpha(\T)$.
\end{theorem}

\begin{proof}
The mapping $u\mapsto \EE_\alpha(u)$ is quadratic and nonnegative, and the term $u\mapsto \lambda\norm{u}_{L^2(\T,\mu)}^2$ is strictly convex because $\lambda>0$. Therefore $\JJ$ is strictly convex on the Hilbert space $\HH_\alpha(\T)$.

We next prove coercivity. By the definition \eqref{eq:def-Hnorm} of the norm,
\[
\EE_\alpha(u)+\lambda\norm{u}_{L^2(\T,\mu)}^2
\ge \min\{1,\lambda\}\,\norm{u}_{\HH_\alpha(\T)}^2 .
\]
By Cauchy--Schwarz and Young's inequality,
\[
\left|\int_{\T}fu\ddmu\right|
\le \norm{f}_{L^2(\T,\mu)}\norm{u}_{L^2(\T,\mu)}
\le \frac{\min\{1,\lambda\}}{4}\norm{u}_{\HH_\alpha(\T)}^2+\frac{1}{\min\{1,\lambda\}}\norm{f}_{L^2(\T,\mu)}^2.
\]
Consequently,
\[
\JJ(u)\ge \frac{\min\{1,\lambda\}}{4}\norm{u}_{\HH_\alpha(\T)}^2-\frac{1}{\min\{1,\lambda\}}\norm{f}_{L^2(\T,\mu)}^2
\to +\infty\qquad\text{as }\norm{u}_{\HH_\alpha(\T)}\to\infty.
\]
Thus $\JJ$ is coercive and bounded from below.

Let $(u_n)$ be a minimizing sequence. Coercivity implies that $(u_n)$ is bounded in $\HH_\alpha(\T)$. Since $\HH_\alpha(\T)$ is a Hilbert space, there exist a subsequence, still denoted by $(u_n)$, and some $u\in\HH_\alpha(\T)$ such that $u_n\rightharpoonup u$ weakly in $\HH_\alpha(\T)$. The quadratic form $u\mapsto \EE_\alpha(u)+\lambda\norm{u}_{L^2(\T,\mu)}^2$ is convex and continuous, hence weakly lower semicontinuous, whereas the linear functional $u\mapsto \int_{\T}fu\ddmu$ is weakly continuous. Therefore,
\[
\JJ(u)\le \liminf_{n\to\infty}\JJ(u_n)=\inf_{w\in\HH_\alpha(\T)}\JJ(w).
\]
Hence $u$ is a minimizer. Uniqueness follows from strict convexity.
\end{proof}

\begin{remark}\label{rem:lambda-degenerates}
The proof of Theorem~\ref{thm:existence-uniqueness} does not use Theorem~\ref{thm:interface-coercivity}: for $\lambda>0$, the zero-order term alone controls the mean level of $u$. However, the coercivity constant $\min\{1,\lambda\}$ degenerates as $\lambda\to0$. Section~\ref{sec:neumann} shows that, thanks to the interface coercivity, the problem remains well posed (in a suitable sense) for $\lambda=0$, and that the estimates are uniform in $\lambda$ for compatible data.
\end{remark}

\begin{theorem}[Variational characterization of weak solutions]\label{thm:weak-var-equivalence}
A function $u\in\HH_\alpha(\T)$ is the unique minimizer of $\JJ$ if and only if it is the unique weak solution of \eqref{eq:weak-solution}.
\end{theorem}

\begin{proof}
Let $u\in\HH_\alpha(\T)$ and $\varphi\in\HH_\alpha(\T)$. For $t\in\R$,
\[
\JJ(u+t\varphi)
=\frac12\EE_\alpha(u+t\varphi)+\frac{\lambda}{2}\norm{u+t\varphi}_{L^2(\T,\mu)}^2-\int_{\T}f(u+t\varphi)\ddmu.
\]
Since the energy is quadratic,
\[
\frac{d}{dt}\Big|_{t=0}\frac12\EE_\alpha(u+t\varphi)=\mathfrak a(u,\varphi),
\]
and similarly,
\[
\frac{d}{dt}\Big|_{t=0}\frac{\lambda}{2}\norm{u+t\varphi}_{L^2(\T,\mu)}^2
=\lambda\int_{\T}u\varphi\ddmu.
\]
The derivative of the linear term is
\[
\frac{d}{dt}\Big|_{t=0}\int_{\T}f(u+t\varphi)\ddmu
=\int_{\T}f\varphi\ddmu.
\]
Hence $u$ is a critical point of $\JJ$ if and only if \eqref{eq:weak-solution} holds. Since $\JJ$ is strictly convex by Theorem~\ref{thm:existence-uniqueness}, it has at most one critical point. The conclusion follows.
\end{proof}

\subsection{Coupled Euler--Lagrange equations}\label{sec:transmission}

The weak formulation can be split into its continuous and discrete components. This yields a coupled system between the two phases, in which the interface integrals act as transmission terms.

\begin{theorem}[Weak Euler--Lagrange identity]\label{thm:weak-EL}
Let $u=(v,a_2,\dots,a_{N+1})\in\HH_\alpha(\T)$ be the minimizer of $\JJ$. Then, for every $\varphi=(\psi,b_2,\dots,b_{N+1})\in\HH_\alpha(\T)$,
\begin{align}\label{eq:split-weak-form}
&\int_0^1\int_0^1
\frac{(v(x)-v(y))(\psi(x)-\psi(y))}{\abs{x-y}^{1+2\alpha}}\,dx\,dy \notag\\
&\quad +2\sum_{k=2}^{N+1}\int_0^1
\frac{(a_k-v(x))(b_k-\psi(x))}{\abs{k-x}^{1+2\alpha}}\,dx \\
&\quad +\sum_{\substack{i,j=2\\ i\neq j}}^{N+1}
\frac{(a_i-a_j)(b_i-b_j)}{\abs{i-j}^{1+2\alpha}}
+\lambda\int_{\T}u\varphi\ddmu
=\int_{\T}f\varphi\ddmu.\notag
\end{align}
\end{theorem}

\begin{proof}
Let $u=(v,a_2,\dots,a_{N+1})\in\HH_\alpha(\T)$ be the minimizer of $\JJ$. By Theorem~\ref{thm:weak-var-equivalence}, $u$ is the unique weak solution of \eqref{eq:weak-solution}. Hence, for every test function
\[
\varphi=(\psi,b_2,\dots,b_{N+1})\in\HH_\alpha(\T),
\]
one has
\begin{equation}\label{eq:proof-thm71-weak}
\mathfrak a(u,\varphi)+\lambda\int_{\T}u\varphi\,\ddmu=\int_{\T}f\varphi\,\ddmu.
\end{equation}
It therefore remains to compute the bilinear form $\mathfrak a(u,\varphi)$ in terms of the continuous and discrete components.

By definition,
\[
\mathfrak a(u,\varphi)=\iint_{\T\times\T}
K_\alpha(x,y)(u(x)-u(y))(\varphi(x)-\varphi(y))\,d\mu(x)\,d\mu(y).
\]
Using the decomposition of the measure
\[
d\mu(x)=\mathbf 1_{[0,1]}(x)\,dx+\sum_{k=2}^{N+1}\delta_k(x),
\]
the product measure $\mu\otimes\mu$ splits into four contributions:
continuous--continuous, continuous--discrete, discrete--continuous, and discrete--discrete. We examine them separately.

For the continuous--continuous part, since $u|_{[0,1]}=v$ and $\varphi|_{[0,1]}=\psi$, we obtain
\[
I_{cc}:=\int_0^1\int_0^1
\frac{(v(x)-v(y))(\psi(x)-\psi(y))}{\abs{x-y}^{1+2\alpha}}\,dx\,dy.
\]
For the continuous--discrete part, using $u(k)=a_k$ and $\varphi(k)=b_k$, we get
\[
I_{cd}:=\sum_{k=2}^{N+1}\int_0^1
\frac{(v(x)-a_k)(\psi(x)-b_k)}{\abs{x-k}^{1+2\alpha}}\,dx.
\]
For the discrete--continuous part, by the symmetry of the kernel and the identity $\abs{k-x}=\abs{x-k}$, we have
\[
I_{dc}:=\sum_{k=2}^{N+1}\int_0^1
\frac{(a_k-v(x))(b_k-\psi(x))}{\abs{k-x}^{1+2\alpha}}\,dx.
\]
Since
\[
(v(x)-a_k)(\psi(x)-b_k)=(a_k-v(x))(b_k-\psi(x)),
\]
it follows that $I_{cd}=I_{dc}$. Therefore the mixed contribution equals
\[
I_{cd}+I_{dc}=2\sum_{k=2}^{N+1}\int_0^1
\frac{(a_k-v(x))(b_k-\psi(x))}{\abs{k-x}^{1+2\alpha}}\,dx.
\]
Finally, for the discrete--discrete part we obtain
\[
I_{dd}:=\sum_{i=2}^{N+1}\sum_{j=2}^{N+1}
K_\alpha(i,j)(a_i-a_j)(b_i-b_j).
\]
By the definition of $K_\alpha$, the diagonal terms have value zero and the remaining terms give
\[
I_{dd}=\sum_{\substack{i,j=2\\ i\neq j}}^{N+1}
\frac{(a_i-a_j)(b_i-b_j)}{\abs{i-j}^{1+2\alpha}}.
\]
Combining the four pieces yields
\begin{align*}
\mathfrak a(u,\varphi)
&=\int_0^1\int_0^1
\frac{(v(x)-v(y))(\psi(x)-\psi(y))}{\abs{x-y}^{1+2\alpha}}\,dx\,dy\\
&\quad+2\sum_{k=2}^{N+1}\int_0^1
\frac{(a_k-v(x))(b_k-\psi(x))}{\abs{k-x}^{1+2\alpha}}\,dx\\
&\quad+\sum_{\substack{i,j=2\\ i\neq j}}^{N+1}
\frac{(a_i-a_j)(b_i-b_j)}{\abs{i-j}^{1+2\alpha}}.
\end{align*}
Substituting this identity into \eqref{eq:proof-thm71-weak} gives exactly \eqref{eq:split-weak-form}. This completes the proof.
\end{proof}

\begin{remark}[Formal pointwise system]\label{rem:pointwise}
Let $u=(v,a_2,\dots,a_{N+1})$ be the weak solution. Assume sufficient
additional regularity and integrability for the following pointwise
expressions to make sense. Then \eqref{eq:split-weak-form} corresponds formally to
\begin{equation}\label{eq:strong-continuous}
2\,\PV\int_0^1\frac{v(x)-v(y)}{\abs{x-y}^{1+2\alpha}}\,dy
+2\sum_{k=2}^{N+1}\frac{v(x)-a_k}{\abs{k-x}^{1+2\alpha}}
+\lambda v(x)=f(x)
\end{equation}
for $x\in(0,1)$, and to
\begin{equation}\label{eq:strong-discrete}
2\int_0^1\frac{a_k-v(x)}{\abs{k-x}^{1+2\alpha}}\,dx
+2\sum_{\substack{j=2\\ j\neq k}}^{N+1}\frac{a_k-a_j}{\abs{k-j}^{1+2\alpha}}
+\lambda a_k=f(k)
\end{equation}
for each $k\in\{2,\dots,N+1\}$.

The first equation is a nonlocal integral equation on the continuous phase, whereas the second one is a finite nonlocal algebraic system on the discrete phase. The interface contribution acts as a transmission term: it inserts the discrete variables into the continuous equation and, conversely, the continuous profile into each discrete balance law.
\end{remark}

\section{The degenerate case \texorpdfstring{$\lambda=0$}{lambda=0}: a Neumann-type problem}\label{sec:neumann}

When $\lambda=0$, the zero-order term disappears and the functional reduces to
\begin{equation}\label{eq:J0}
\JJ_0(u):=\frac12\EE_\alpha(u)-\int_{\T}fu\ddmu,\qquad u\in\HH_\alpha(\T).
\end{equation}
Since $\EE_\alpha(u+c)=\EE_\alpha(u)$ for every constant $c$, the energy is degenerate on constants and $\JJ_0$ is not coercive on $\HH_\alpha(\T)$. In this regime, the interface coercivity of Theorem~\ref{thm:interface-coercivity} becomes the decisive ingredient.

Throughout this section, a constant $c\in\R$ is identified with the function $(c,c,\dots,c)\in\HH_\alpha(\T)$. We set
\[
\avg{u}:=\frac{1}{\mu(\T)}\int_{\T}u\ddmu=\frac{1}{N+1}\Big(\bar v+\sum_{k=2}^{N+1}a_k\Big),
\]
\[
\HH_{\alpha,0}(\T):=\Big\{u\in\HH_\alpha(\T):\ \int_{\T}u\ddmu=0\Big\}.
\]

\begin{lemma}[Kernel of the energy]\label{lem:kernel}
For $u\in\HH_\alpha(\T)$, one has $\EE_\alpha(u)=0$ if and only if $u$ is constant on $\T$.
\end{lemma}

\begin{proof}
If $u$ is constant, then $\EE_\alpha(u)=0$ by \eqref{eq:def-energy}. Conversely, if $\EE_\alpha(u)=0$, Corollary~\ref{cor:hybrid-poincare} gives $\norm{v-\bar v}_{L^2(0,1)}=0$ and $a_k=\bar v$ for all $k$, so that $u\equiv\bar v$ $\mu$-a.e.
\end{proof}

\begin{theorem}[Global hybrid Poincar\'e inequality]\label{thm:global-poincare}
For every $u\in\HH_\alpha(\T)$,
\begin{equation}\label{eq:global-poincare}
\norm{u-\avg{u}}_{L^2(\T,\mu)}^2\le 4C_P\,\EE_\alpha(u),
\end{equation}
where $C_P$ is the constant of Corollary~\ref{cor:hybrid-poincare}. In particular, $\norm{u}_{L^2(\T,\mu)}^2\le 4C_P\,\EE_\alpha(u)$ for every $u\in\HH_{\alpha,0}(\T)$.
\end{theorem}

\begin{proof}
Let $w:=u-\avg{u}=(z,d_2,\dots,d_{N+1})$. Then $w\in\HH_{\alpha,0}(\T)$ and $\EE_\alpha(w)=\EE_\alpha(u)$. Since $\int_\T w\ddmu=0$, we have $\bar z+\sum_{k}d_k=0$, that is,
\[
(N+1)\,\bar z=-\sum_{k=2}^{N+1}(d_k-\bar z).
\]
By the Cauchy--Schwarz inequality,
\[
\abs{\bar z}^2\le \frac{N}{(N+1)^2}\sum_{k=2}^{N+1}\abs{d_k-\bar z}^2.
\]
Using $\norm{z}_{L^2(0,1)}^2=\norm{z-\bar z}_{L^2(0,1)}^2+\abs{\bar z}^2$ and $\abs{d_k}^2\le 2\abs{d_k-\bar z}^2+2\abs{\bar z}^2$, we obtain
\begin{align*}
\norm{w}_{L^2(\T,\mu)}^2
&\le \norm{z-\bar z}_{L^2(0,1)}^2+2\sum_{k=2}^{N+1}\abs{d_k-\bar z}^2+(2N+1)\abs{\bar z}^2\\
&\le \norm{z-\bar z}_{L^2(0,1)}^2+\Big(2+\frac{N(2N+1)}{(N+1)^2}\Big)\sum_{k=2}^{N+1}\abs{d_k-\bar z}^2\\
&\le 4\Big(\norm{z-\bar z}_{L^2(0,1)}^2+\sum_{k=2}^{N+1}\abs{d_k-\bar z}^2\Big)
\le 4C_P\,\EE_\alpha(w),
\end{align*}
where we used $N(2N+1)<2(N+1)^2$ and Corollary~\ref{cor:hybrid-poincare}.
\end{proof}

\begin{corollary}\label{cor:H0-hilbert}
$\HH_{\alpha,0}(\T)$ is a closed subspace of $\HH_\alpha(\T)$, and $\mathfrak a$ is an inner product on $\HH_{\alpha,0}(\T)$ whose norm is equivalent to $\norm{\cdot}_{\HH_\alpha(\T)}$:
\[
\EE_\alpha(u)\le \norm{u}_{\HH_\alpha(\T)}^2\le (1+4C_P)\,\EE_\alpha(u)\qquad (u\in\HH_{\alpha,0}(\T)).
\]
\end{corollary}

\begin{proof}
The linear functional $u\mapsto\int_\T u\ddmu$ is continuous on $\HH_\alpha(\T)$, since $\abs{\int_\T u\ddmu}\le (N+1)^{1/2}\norm{u}_{L^2(\T,\mu)}$; its kernel $\HH_{\alpha,0}(\T)$ is therefore closed. The two-sided estimate follows from \eqref{eq:def-Hnorm} and Theorem~\ref{thm:global-poincare}.
\end{proof}

We now consider the weak problem with $\lambda=0$: find $u\in\HH_\alpha(\T)$ such that
\begin{equation}\label{eq:neumann-weak}
\mathfrak a(u,\varphi)=\int_\T f\varphi\ddmu\qquad\text{for every }\varphi\in\HH_\alpha(\T).
\end{equation}

\begin{theorem}[Neumann-type hybrid problem]\label{thm:neumann}
Let $f\in L^2(\T,\mu)$.
\begin{enumerate}[label=\textup{(\roman*)}]
\item Problem \eqref{eq:neumann-weak} has a solution if and only if
\begin{equation}\label{eq:compatibility}
\int_\T f\ddmu=\int_0^1 f(x)\,dx+\sum_{k=2}^{N+1}f(k)=0.
\end{equation}
If \eqref{eq:compatibility} fails, $\JJ_0$ is unbounded from below.
\item Under \eqref{eq:compatibility}, problem \eqref{eq:neumann-weak} has a unique solution $u_0\in\HH_{\alpha,0}(\T)$, and the set of all solutions is $\{u_0+c:\ c\in\R\}$.
\item The solution $u_0$ satisfies
\[
\EE_\alpha(u_0)\le 4C_P\norm{f}_{L^2(\T,\mu)}^2,
\qquad
\norm{u_0}_{\HH_\alpha(\T)}^2\le 4C_P(1+4C_P)\norm{f}_{L^2(\T,\mu)}^2.
\]
\item Under \eqref{eq:compatibility}, the solutions of \eqref{eq:neumann-weak} are exactly the minimizers of $\JJ_0$ on $\HH_\alpha(\T)$, and $u_0$ is the unique minimizer of $\JJ_0$ on $\HH_{\alpha,0}(\T)$.
\end{enumerate}
\end{theorem}

\begin{proof}
Set $\ell(\varphi):=\int_\T f\varphi\ddmu$.

(i) Necessity. If $u$ solves \eqref{eq:neumann-weak}, choosing the constant test function $\varphi\equiv1$ gives $\ell(1)=\mathfrak a(u,1)=0$, which is \eqref{eq:compatibility}. Moreover, if \eqref{eq:compatibility} fails, then $\JJ_0(c)=-c\int_\T f\ddmu$ is unbounded from below as $c$ ranges over $\R$. Sufficiency follows from (ii).

(ii) By Theorem~\ref{thm:global-poincare}, for $\varphi\in\HH_{\alpha,0}(\T)$,
\[
\abs{\ell(\varphi)}\le\norm{f}_{L^2(\T,\mu)}\norm{\varphi}_{L^2(\T,\mu)}\le (4C_P)^{1/2}\norm{f}_{L^2(\T,\mu)}\,\mathfrak a(\varphi,\varphi)^{1/2}.
\]
By Corollary~\ref{cor:H0-hilbert} and the Riesz representation theorem, there exists a unique $u_0\in\HH_{\alpha,0}(\T)$ such that $\mathfrak a(u_0,\varphi)=\ell(\varphi)$ for every $\varphi\in\HH_{\alpha,0}(\T)$. For an arbitrary $\varphi\in\HH_\alpha(\T)$, write $\varphi=(\varphi-\avg{\varphi})+\avg{\varphi}$. Since $\mathfrak a(u_0,c)=0$ and, by \eqref{eq:compatibility}, $\ell(c)=0$ for every constant $c$, it follows that $\mathfrak a(u_0,\varphi)=\ell(\varphi)$. Hence $u_0$ solves \eqref{eq:neumann-weak}, and so does $u_0+c$ for every $c\in\R$. Conversely, if $u$ is a solution, then $w:=u-u_0$ satisfies $\mathfrak a(w,\varphi)=0$ for all $\varphi$; taking $\varphi=w$ gives $\EE_\alpha(w)=0$, so $w$ is constant by Lemma~\ref{lem:kernel}. Since the only constant in $\HH_{\alpha,0}(\T)$ is $0$, $u_0$ is the unique solution in $\HH_{\alpha,0}(\T)$.

(iii) Taking $\varphi=u_0$ in \eqref{eq:neumann-weak} and using the bound on $\ell$ from (ii),
\[
\EE_\alpha(u_0)=\ell(u_0)\le (4C_P)^{1/2}\norm{f}_{L^2(\T,\mu)}\,\EE_\alpha(u_0)^{1/2},
\]
which gives the first estimate; the second follows from Corollary~\ref{cor:H0-hilbert}.

(iv) For $u,\varphi\in\HH_\alpha(\T)$, the quadratic structure of $\JJ_0$ gives
\[
\JJ_0(u+\varphi)=\JJ_0(u)+\big[\mathfrak a(u,\varphi)-\ell(\varphi)\big]+\tfrac12\EE_\alpha(\varphi).
\]
If $u$ solves \eqref{eq:neumann-weak}, the bracket vanishes and $\JJ_0(u+\varphi)\ge\JJ_0(u)$, so $u$ is a minimizer. Conversely, if $u$ is a minimizer, then $t\mapsto\JJ_0(u+t\varphi)$ has a minimum at $t=0$, and differentiating as in the proof of Theorem~\ref{thm:weak-var-equivalence} yields \eqref{eq:neumann-weak}. By (ii), the minimizers in $\HH_{\alpha,0}(\T)$ reduce to $u_0$.
\end{proof}

\begin{remark}[The transmission mechanism]\label{rem:two-conditions}
The single compatibility condition \eqref{eq:compatibility} is a genuine transmission effect. Suppose the interface term were removed, that is, consider the decoupled energy $\widetilde{\EE}_\alpha:=\EE_{\cc}+\EE_{\dd}$. Its kernel is the two-dimensional space $\{(c_1,c_2,\dots,c_2):\ c_1,c_2\in\R\}$, and the corresponding Neumann-type problem splits into two independent problems, each requiring its own condition:
\[
\int_0^1 f(x)\,dx=0\qquad\text{and}\qquad\sum_{k=2}^{N+1}f(k)=0.
\]
With the interface term, the kernel reduces to global constants (Lemma~\ref{lem:kernel}) and only the global balance \eqref{eq:compatibility} is required: a source on one phase can be compensated by a sink on the other. This is exactly the content of Theorem~\ref{thm:interface-coercivity}, which is used in an essential way in Theorem~\ref{thm:global-poincare}. Formally, the solution satisfies the pointwise system \eqref{eq:strong-continuous}--\eqref{eq:strong-discrete} with $\lambda=0$; the algebraic decomposition used in the proof of Theorem~\ref{thm:weak-EL} also applies to $u_0$.
\end{remark}

We finally show that the degenerate problem is the limit of the regularized ones, with estimates that are uniform in $\lambda$.

\begin{proposition}[Vanishing zero-order term]\label{prop:lambda-limit}
Let $f\in L^2(\T,\mu)$ satisfy \eqref{eq:compatibility}. For $\lambda>0$, let $u_\lambda$ be the weak solution of \eqref{eq:weak-solution} given by Theorems~\ref{thm:existence-uniqueness} and~\ref{thm:weak-var-equivalence}, and let $u_0\in\HH_{\alpha,0}(\T)$ be given by Theorem~\ref{thm:neumann}. Then $u_\lambda\in\HH_{\alpha,0}(\T)$ and
\begin{align*}
\norm{u_\lambda-u_0}_{\HH_\alpha(\T)}
&\le 2\lambda\big(C_P(1+4C_P)\big)^{1/2}\norm{u_0}_{L^2(\T,\mu)}\\
&\le 8\lambda\,C_P\big(C_P(1+4C_P)\big)^{1/2}\norm{f}_{L^2(\T,\mu)}.
\end{align*}
In particular, $u_\lambda\to u_0$ in $\HH_\alpha(\T)$ as $\lambda\to0$, and $\sup_{0<\lambda\le1}\norm{u_\lambda}_{\HH_\alpha(\T)}\le C\norm{f}_{L^2(\T,\mu)}$ with $C$ depending only on $\alpha$ and $N$.
\end{proposition}

\begin{proof}
Choosing $\varphi\equiv1$ in \eqref{eq:weak-solution} gives $\lambda\int_\T u_\lambda\ddmu=\int_\T f\ddmu=0$, hence $u_\lambda\in\HH_{\alpha,0}(\T)$. Set $w:=u_\lambda-u_0\in\HH_{\alpha,0}(\T)$. Subtracting \eqref{eq:neumann-weak} from \eqref{eq:weak-solution} yields
\[
\mathfrak a(w,\varphi)+\lambda\int_\T u_\lambda\varphi\ddmu=0\qquad\text{for every }\varphi\in\HH_\alpha(\T).
\]
Taking $\varphi=w$ and writing $u_\lambda=w+u_0$,
\[
\EE_\alpha(w)+\lambda\norm{w}_{L^2(\T,\mu)}^2
=-\lambda\int_\T u_0w\ddmu
\le\lambda\norm{u_0}_{L^2(\T,\mu)}\norm{w}_{L^2(\T,\mu)}.
\]
By Theorem~\ref{thm:global-poincare}, $\norm{w}_{L^2(\T,\mu)}\le(4C_P)^{1/2}\EE_\alpha(w)^{1/2}$, so that
\[
\EE_\alpha(w)^{1/2}\le\lambda(4C_P)^{1/2}\norm{u_0}_{L^2(\T,\mu)}.
\]
Corollary~\ref{cor:H0-hilbert} gives the first estimate. The second follows from
\[
\norm{u_0}_{L^2(\T,\mu)}^2\le 4C_P\,\EE_\alpha(u_0)\le16C_P^2\norm{f}_{L^2(\T,\mu)}^2,
\]
by Theorems~\ref{thm:global-poincare} and~\ref{thm:neumann}(iii). The uniform bound follows from the triangle inequality.
\end{proof}

\begin{remark}
For compatible data, Proposition~\ref{prop:lambda-limit} improves the bound of Theorem~\ref{thm:existence-uniqueness}, whose constant behaves like $\min\{1,\lambda\}^{-1}$: the interface coercivity yields estimates that do not deteriorate as $\lambda\to0$.
\end{remark}

\section{Numerical experiments and structural tests}\label{sec:numerical}

This section investigates the quantitative predictions of
Theorem~\ref{thm:interface-coercivity} and the Neumann-type theory of
Section~\ref{sec:neumann}. The computations are tests of finite-dimensional
representations, not substitutes for continuous proofs. They distinguish
an $L^2$ spectral gap from the full fractional coercivity of $M(u)$,
and include quadrature and resolution checks.

\subsection{A two-node hybrid example}

To make the coupling explicit, take
$\T=[0,1]\cup\{2,3\}$, $u=(v,a,b)$, and
$d\mu=\mathbf 1_{[0,1]}\,dx+\delta_2+\delta_3$. The energy is
\begin{align}\label{eq:example-energy}
\EE_\alpha(u)
={}&\int_0^1\!\int_0^1
\frac{|v(x)-v(y)|^2}{|x-y|^{1+2\alpha}}\,dx\,dy \notag\\
&+2\int_0^1\frac{|a-v(x)|^2}{|2-x|^{1+2\alpha}}\,dx
+2\int_0^1\frac{|b-v(x)|^2}{|3-x|^{1+2\alpha}}\,dx
+2|a-b|^2.
\end{align}
For $f\in L^2(0,1)$, $F,G\in\R$, the forcing functional on
$\HH_\alpha(\T)$ is
\begin{equation}\label{eq:example-functional}
\JJ(u)=\frac12\EE_\alpha(u)+\frac{\lambda}{2}
\left(\int_0^1|v|^2\,dx+a^2+b^2\right)
-\int_0^1 fv\,dx-Fa-Gb.
\end{equation}
Theorems~\ref{thm:existence-uniqueness} and
\ref{thm:weak-EL} apply directly; hence their existence and
Euler--Lagrange proofs need not be repeated for this specialization.
The continuous weak equation is, for every $\psi\in H^\alpha(0,1)$,
\begin{align}\label{eq:example-continuous}
&\int_0^1\!\int_0^1
\frac{(v(x)-v(y))(\psi(x)-\psi(y))}{|x-y|^{1+2\alpha}}\,dx\,dy
\notag\\
&\quad +2\int_0^1\frac{(v(x)-a)\psi(x)}{|2-x|^{1+2\alpha}}\,dx
+2\int_0^1\frac{(v(x)-b)\psi(x)}{|3-x|^{1+2\alpha}}\,dx
+\lambda\int_0^1v\psi\,dx
=\int_0^1 f\psi\,dx .
\end{align}
The discrete balance equations are
\begin{align}
2\int_0^1\frac{a-v(x)}{|2-x|^{1+2\alpha}}\,dx
+2(a-b)+\lambda a&=F,\label{eq:example-discrete-a}\\
2\int_0^1\frac{b-v(x)}{|3-x|^{1+2\alpha}}\,dx
+2(b-a)+\lambda b&=G.\label{eq:example-discrete-b}
\end{align}
Thus both nodal values communicate with the entire continuous profile.
When $\lambda=0$, Theorem~\ref{thm:neumann} yields the
\emph{single} compatibility condition
$\int_0^1 f(x)\,dx+F+G=0$, rather than one condition per phase.

\subsection{Galerkin discretization and reproducibility}

Let $0=x_0<x_1<\cdots<x_m=1$ be the uniform partition with
$h=1/m$, and let $V_h$ be the continuous piecewise-affine space on
$[0,1]$. The discrete trial space is
\[
\mathcal H_{\alpha,h}=V_h\times\mathbb R^N,
\qquad
u_h=(v_h,a_2,\dots,a_{N+1}),\quad
v_h=\sum_{j=0}^{m}c_j\varphi_j.
\]
Writing $U=(c_0,\dots,c_m,a_2,\dots,a_{N+1})^\top$, we assemble
symmetric matrices $K_{cc,h}$, $K_{cd,h}$, $K_{dd,h}$ and $B_h$ with
\[
\begin{aligned}
E_{cc}(v_h)&=U^\top K_{cc,h}U,
& 2E_{cd}(v_h,a)&=U^\top K_{cd,h}U,\\
E_{dd}(a)&=U^\top K_{dd,h}U,
& \|u_h\|_{L^2(\T,\mu)}^2&=U^\top B_hU.
\end{aligned}
\]
Thus $K_h=K_{cc,h}+K_{cd,h}+K_{dd,h}$ represents the 
\emph{entire} ordered-pair energy of~\eqref{eq:def-energy}.
The mass matrix is the standard P1 mass matrix on the interval,
augmented by an identity matrix on the $N$ unit-mass atoms; this is the
prescribed Lebesgue--Dirac measure~\eqref{eq:def-mu}, not the
canonical time-scale $\Delta$-measure.

On equal P1 cells, the weak diagonal singularity in $E_{cc}$ is
integrated analytically. Adjacent-cell contributions use a Duffy
transformation with the radial integration evaluated exactly; the
remaining smooth continuous interactions use tensor-product Gauss--Legendre
quadrature. Mixed contributions are integrated by one-dimensional
Gauss--Legendre quadrature and discrete--discrete interactions exactly.
The default quadrature order is $10$. For the general parameter scan we use
\[
\alpha\in\{0.35,0.65,0.85\},\qquad
N\in\{2,3,6\},\qquad m\in\{12,24,48\}.
\]
The auxiliary tests of the Neumann problem and regularization use
$N=3$, $m=24$, and the same three values of $\alpha$. All reported
quantities are computed from the reproducible Python and Kaggle scripts
accompanying this manuscript, with machine-readable CSV outputs.
The files can be generated without external datasets.

The assembly has two independent diagnostics. First, the vector of global
constants $\mathbf 1$ satisfies $K_h\mathbf 1\approx 0$ up to roundoff.
Second, taking the affine test function $v(x)=x$ and constant atomic values
identical across the nodes, the continuous--continuous energy has the
exact benchmark
\begin{equation}\label{eq:affine-benchmark}
\int_0^1\!\int_0^1\frac{|x-y|^2}{|x-y|^{1+2\alpha}}\,dx\,dy
=\frac{2}{(2-2\alpha)(3-2\alpha)}.
\end{equation}
The maximum observed relative error in this benchmark, over the tested
fractional orders and meshes, is $2.4\times10^{-13}$; this validates that
particular assembly identity and is not an error estimate for the
associated continuous eigenvalue problem.

\subsection{Spectral constants and the hybrid Poincar\'e bound}

Let $\mathbf m$ be the vector representing $\int_\T u_h\,d\mu$.
All generalized eigenvalues below are computed on the codimension-one
subspace $\mathbf m^\top U=0$.
The first $L^2$ gap is the Rayleigh minimum
\begin{equation}\label{eq:discrete-gap}
\gamma_h=\min_{\substack{U\ne0\\\mathbf m^\top U=0}}
\frac{U^\top K_hU}{U^\top B_hU}.
\end{equation}
Let $S_h$ denote the quadratic matrix representing the full $M(u_h)$
from Theorem~\ref{thm:interface-coercivity}. The extremal generalized eigenvalues
of $(K_h,S_h)$ on the same subspace give numerical estimates of the
best comparison constants. Because constants are null modes of both forms,
this restriction removes the common degeneracy without changing the
nonzero Rayleigh quotients.

\begin{table}[htbp]
\centering\small
\caption{Spectral tests for $N=3$ and $m=24$. Here $\gamma_h$ is
\eqref{eq:discrete-gap}; $\gamma_{\mathrm{th}}=(4C_P)^{-1}$ is the
rigorous lower bound from Theorem~\ref{thm:global-poincare};
$c_{1,h}$ and $c_{2,h}$ are the smallest and largest computed ratios
$E_\alpha(u_h)/M(u_h)$.}
\label{tab:num-spectral}
\begin{tabular}{@{}ccccc@{}}
\toprule
$\alpha$&$\gamma_h$&$\gamma_{\mathrm{th}}$&$c_{1,h}$&$c_{2,h}$\\
\midrule
$0.35$&$2.066437$&$0.047366$&$0.540539$&$6.540058$\\
$0.65$&$1.466825$&$0.020617$&$0.385354$&$6.363458$\\
$0.85$&$1.210003$&$0.011842$&$0.317987$&$6.285561$\\
\bottomrule
\end{tabular}
\end{table}

\begin{figure}[htbp]
\centering
\includegraphics[width=.81\textwidth]{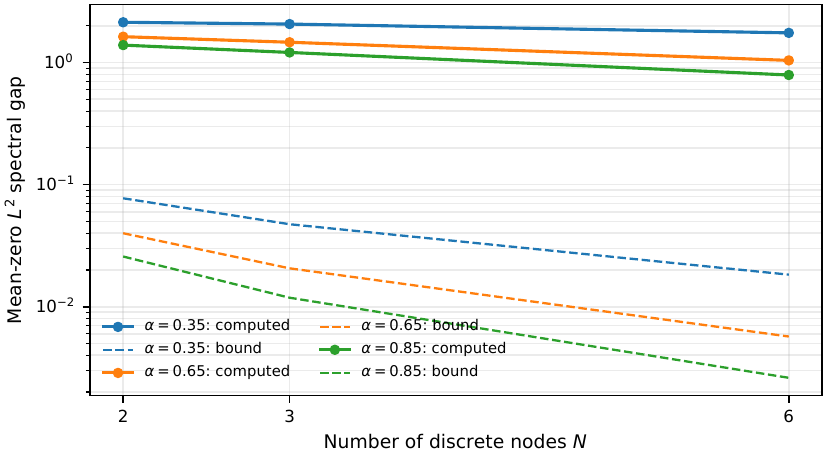}
\caption{Mean-zero Galerkin $L^2$ spectral gap (solid curves) versus the
analytical lower bound (dashed curves), at $m=24$.
The comparison concerns the selected measure and finite-dimensional
spaces; the plots do not provide certified eigenvalue errors.}
\label{fig:num-spectral}
\end{figure}

For all $27$ combinations of $\alpha$, $N$, and $m$ in the parameter
scan, the computed ratios obey the explicit theoretical inequalities.
The smallest observed ratio $\gamma_h/\gamma_{\mathrm{th}}$ is
$27.678$, showing that the universal analytical estimate is conservative
on the tested geometries. In the representative case $\alpha=0.65$,
$N=3$, $m=24$, the theoretical constants can be taken as
$C_1=0.082469$ and $C_2=11.624505$, while the computed extremal
quotients are $c_{1,h}=0.385354$ and $c_{2,h}=6.363458$.
This illustrates the gap between an explicit general estimate and its
numerical counterpart; it does not identify sharp continuous constants.
For the same configuration, increasing the Gauss order from $6$ to
$10$ and $14$ changes the computed gap by less than $10^{-12}$.

\subsection{Ablation study and necessity of energy components}

We repeat the spectral calculations for five forms:
$E_{cc}+2E_{cd}+E_{dd}$ (full), $E_{cc}+E_{dd}$ (without interface),
$E_{cc}+2E_{cd}$ (without discrete--discrete energy),
$2E_{cd}+E_{dd}$ (without continuous--continuous energy), and
$2E_{cd}$ (interface alone). Table~\ref{tab:num-ablation} gives the
representative results for $\alpha=0.65$, $N=3$, $m=24$.
The final column is the minimum quotient of the ablated energy by
$M(u_h)$ on the mean-zero subspace; unlike the $L^2$ gap, it tests the
fractional regularity component in Theorem~\ref{thm:interface-coercivity}.

\begin{table}[htbp]
\centering\small
\caption{Ablation at $\alpha=0.65$, $N=3$, $m=24$.
A numerical zero is reported as $0$ when below the indicated
matrix precision. ``Extra zero'' counts nonconstant modes in the
mean-zero subspace.}
\label{tab:num-ablation}
\begin{tabular}{@{}lccc@{}}
\toprule
Energy variant & $\gamma_h$ & $\min E/M$ & Extra zero\\
\midrule
Full energy & $1.466825$ & $0.385354$ & $0$\\
Without interface & $0$ & $0$ & $1$\\
Without $E_{dd}$ & $0.141160$ & $0.115054$ & $0$\\
Without $E_{cc}$ & $0.656833$ & $4.34\times10^{-4}$ & $0$\\
Interface only & $0.141102$ & $4.34\times10^{-4}$ & $0$\\
\bottomrule
\end{tabular}
\end{table}

\begin{figure}[htbp]
\centering
\includegraphics[width=.79\textwidth]{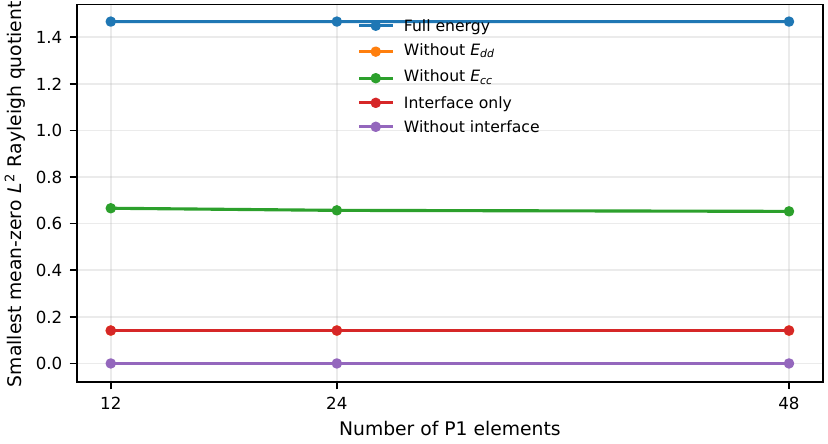}
\caption{The smallest mean-zero $L^2$ Rayleigh quotient under systematic
removal of energy components. Removing the interface yields a new
nonconstant zero mode, whereas removing $E_{cc}$ does not eliminate the
$L^2$ gap.}
\label{fig:num-ablation-L2}
\end{figure}

\begin{figure}[htbp]
\centering
\includegraphics[width=.79\textwidth]{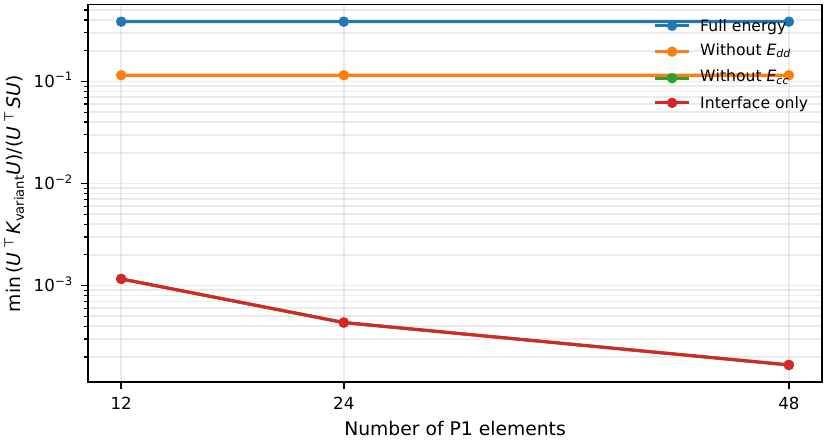}
\caption{The minimum energy-to-$M$ ratio under refinement, for
$\alpha=0.65$ and $N=3$. Removing $E_{cc}$ leads to a progressively
vanishing ratio even though the corresponding $L^2$ gap remains positive.
The no-interface variant is omitted because its ratio is zero.}
\label{fig:num-ablation-Halpha}
\end{figure}

The no-interface form has two independent global-level degrees of
freedom, one on the interval and one on the discrete nodes; its zero
mode remaining after removal of the total constant appears as the
extra null vector in Table~\ref{tab:num-ablation}. In contrast,
$E_{cc}+2E_{cd}$ retains the coercivity required by
Theorem~\ref{thm:interface-coercivity}: the discrete--discrete term $E_{dd}$
strengthens the numerical gap but is not necessary for the qualitative
inequality. At $\alpha=0.65$, the computed $L^2$ gap decreases from
$1.466825$ to $0.141160$ when $E_{dd}$ is removed.

Removing $E_{cc}$ leads to a different phenomenon. The $L^2$ gap
remains positive, but the ratio to $M$ decreases from
$1.16\times10^{-3}$ at $m=12$ to $4.34\times10^{-4}$ at $m=24$
and $1.67\times10^{-4}$ at $m=48$. This is the expected signature
of failure to control the fractional seminorm on progressively richer
spaces. The mathematical distinctions behind these observations are
made explicit in the following proposition.

\begin{proposition}[Structural role of each energy component]
\label{prop:ablation-structure}
Fix $0<\alpha<1$ and $N\ge2$, and let $M(u)$ be as in
Theorem~\ref{thm:interface-coercivity}.
\begin{enumerate}[label=\textup{(\roman*)}]
\item The reduced energy $E_{cc}(v)+2E_{cd}(v,a)$ is equivalent to
$M(u)$ on $\HH_\alpha(\T)$.
\item Without the interface, the kernel of $E_{cc}+E_{dd}$
is two-dimensional: its functions are constant on each phase, with
independent values on the two phases.
\item The energy $2E_{cd}(v,a)+E_{dd}(a)$ controls the deviation
of $u$ from global constants in $L^2(\T,\mu)$, but there is no
constant $c>0$ for which
$2E_{cd}(u)+E_{dd}(u)\ge cM(u)$ holds for all
$u\in\HH_\alpha(\T)$.
\end{enumerate}
\end{proposition}

\begin{proof}
For (i), the lower estimate in the proof of
Theorem~\ref{thm:interface-coercivity} uses only $E_{cc}$ and $2E_{cd}$.
The upper estimate follows directly from Lemmas~\ref{lem:interface-estimate}
and~\ref{lem:mean-decomposition}. For (ii), vanishing of $E_{cc}$ forces
$v$ to be constant a.e., and vanishing of $E_{dd}$ forces all
atomic values to coincide, with no relation between the two constants.
For (iii), Lemmas~\ref{lem:interface-estimate} and
\ref{lem:mean-decomposition} imply control of
$\|v-\bar v\|_{L^2(0,1)}^2+\sum_k|a_k-\bar v|^2$;
a global $L^2$ Poincar\'e estimate follows as in
Theorem~\ref{thm:global-poincare}.
However, take $a_k=0$ and $v_n(x)=\sin(2\pi nx)$.
Then $\bar v_n=0$, the mixed energy is bounded uniformly in $n$
because the interval and atoms are separated by a positive distance,
and $E_{dd}=0$. On the other hand,
$[v_n]_{H^\alpha(0,1)}^2\to\infty$ as $n\to\infty$
(for example, restrict the increment representation to
$|x-y|\asymp n^{-1}$). Consequently the quotient of the reduced
energy by $M(u_n)$ tends to zero.
\end{proof}

\subsection{Global Neumann compatibility and the vanishing-zero-order limit}

For a direct test of Theorem~\ref{thm:neumann}, we set
$f(x)=1$ on $(0,1)$, $f(2)=-1$, and $f(k)=0$ at the other atoms.
The total load is zero although the continuous and discrete balances
are $+1$ and $-1$, respectively. The mean-zero discrete Neumann system
$K_hU=F_h$ is solved by eliminating one atomic degree of freedom.
For $N=3$, $m=24$, and $\alpha=0.35,0.65,0.85$, the relative
Euclidean residuals are $1.2\times10^{-15}$,
$2.8\times10^{-14}$, and $8.6\times10^{-14}$, respectively.
The decoupled energy $E_{cc}+E_{dd}$ cannot satisfy the same load:
testing against the independent phase constants produces the two
nonzero balances $+1$ and $-1$.
This is a finite-dimensional illustration of the single global
compatibility condition, not an independent derivation of the theorem.

For compatible data the regularized Galerkin solutions are computed
from $(K_h+\lambda B_h)U_{h,\lambda}=F_h$ with the mean fixed to zero.
We use $\lambda\in\{10^{-4},3\cdot10^{-4},10^{-3},3\cdot10^{-3},
10^{-2},3\cdot10^{-2},10^{-1}\}$.
At fixed $N=3$, $m=24$, least-squares log--log slopes fitted on the
four smallest $\lambda$ values are $0.999631$, $0.999464$, and
$0.999342$ for $\alpha=0.35$, $0.65$, and $0.85$, respectively.
The slopes are consistent with the $O(\lambda)$ estimate in
Proposition~\ref{prop:lambda-limit} for these discrete problems.

\begin{figure}[htbp]
\centering
\includegraphics[width=.81\textwidth]{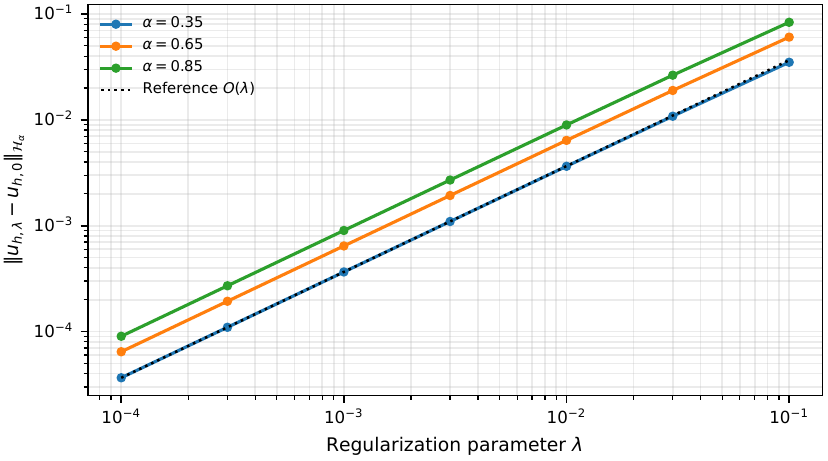}
\caption{Error between regularized and mean-zero Neumann Galerkin
solutions in the norm induced by $K_h+B_h$, with a reference
$O(\lambda)$ line. The mesh is held fixed; no statement on simultaneous
$h\to0$, $\lambda\to0$ is inferred.}
\label{fig:num-regularization}
\end{figure}

For globally incompatible data, take $f\equiv1$ on the entire
hybrid domain. Testing the regularized equation against constants
implies the exact identity
\[
\lambda\int_\T u_\lambda\,d\mu=\mu(\T),
\qquad\text{or equivalently}\qquad
\avg{u_\lambda}=\lambda^{-1}.
\]
With $\alpha=0.65$, $N=3$, $m=24$ and $\lambda$ decreasing from
$0.1$ to $0.003$, the observed mean increases from approximately
$10$ to $333.33$. The deviation in
$\lambda\avg{u_{h,\lambda}}$ from its exact value $1$ is below
$6\times10^{-12}$ over these runs. This divergence of the mean is
compatible with the failure of the Neumann solvability condition
and does not contradict bounded behavior on the mean-zero component.

\subsection{Scope and limitations of the numerical evidence}

The conforming $P_1$ discretization used here is related in spirit to the fractional Neumann finite-element treatment of \cite{BersetcheBorthagaray2022} and to the disjoint-domain coupling schemes of \cite{BersetcheOtarolaQuero2026}. We do not obtain their a priori discretization error estimates or alternating-scheme convergence results. Our computations instead isolate the effects of the finite atomic component and of deleting selected energy terms. This is a distinction in the \emph{question asked} and in the model geometry, not a claim of numerical-method priority.

The numerical evidence distinguishes three logically different facts:
(1) the interface removes the relative-constant mode and ensures global
$L^2$ coupling; (2) the continuous--continuous energy provides the
fractional regularity control absent in the interface term alone; and
(3) the discrete--discrete energy is quantitatively beneficial, but
qualitatively dispensable for the coercivity of
Theorem~\ref{thm:interface-coercivity}. Proposition~\ref{prop:ablation-structure}
provides independent mathematical support for these statements.

The values in Tables~\ref{tab:num-spectral}--\ref{tab:num-ablation}
are Galerkin generalized eigenvalues on finite-dimensional spaces;
the scan covers only the listed parameters and meshes. A quadrature
sensitivity check on one representative configuration and an affine
assembly benchmark are not certified discretization error bounds.
In particular, the current study does not prove convergence of the
Galerkin eigenvalues, establish sharp continuous coercivity constants,
or address the limit $N\to\infty$. A systematic spectral convergence
analysis and sharper dependence on geometry remain separate problems.

\section{Conclusion and perspectives}

We have developed the variational and numerical analysis of a particular continuous--atomic Gagliardo energy under a prescribed Lebesgue--Dirac measure. Building on the interface-coercivity mechanism of the earlier preprint \cite{AbbasAzzouz2026Transmission}, we have investigated its consequences for the degenerate regime and the distinct roles of the three energy contributions. For $\lambda>0$, the associated variational problem has a unique weak solution. For $\lambda=0$, the interface coercivity is indispensable: the resulting Neumann-type problem is solvable under a single global compatibility condition, instead of one condition per phase, and its solution is the limit of the regularized ones with a linear rate in $\lambda$.

The weak formulation obtained in this way naturally splits into a nonlocal equation on the continuous component and discrete balance relations on the nodal part, with the interface integrals transmitting information between the two phases. The variational example and the reproducible Galerkin experiments illustrate the coupled structure, the spectral effect of the interface, and the global compatibility condition. Systematic ablations further establish that the discrete--discrete energy strengthens coercivity but is not needed for its qualitative validity, whereas the continuous--continuous term is indispensable for fractional regularity control.

A natural continuation of this work would be to consider more general finite discrete sets, to track the dependence of the constants on $N$ (in particular in the limit $N\to\infty$), to study reduced Galerkin-type approximations of the coupled problem, and to extend the analysis to nonlinear energies with a similar coupled continuous-discrete structure.

\section*{Declarations}
\subsection*{Competing interests} The authors declare that they have no known competing financial interests or personal relationships that could have appeared to influence the work reported in this paper.

\subsection*{Funding sources}
Funding acknowledgements will be finalized after confirmation from all coauthors.

\subsection*{Declaration of Generative AI}
The authors used ChatGPT to assist with language editing and numerical-code drafting. The authors are responsible for verifying the mathematical statements, computer code, and reported calculations before journal submission. 
\subsection*{Data availability statement} The numerical data in this article are generated computationally. Source code, a standalone Kaggle notebook, parameter configurations, and the CSV tables underlying the figures are provided in the accompanying reproducibility archive; no external dataset is required.

\end{document}